\documentclass[11pt]{amsart}
\usepackage{amsmath,amssymb,amsthm,mathtools}
\usepackage{enumitem, subcaption}
\usepackage{hyperref}
\hypersetup{colorlinks=true,linkcolor=blue,citecolor=blue,urlcolor=blue}

\theoremstyle{plain}
\newtheorem{theorem}{Theorem}
\newtheorem{proposition}[theorem]{Proposition}
\newtheorem{lemma}[theorem]{Lemma}
\newtheorem{corollary}[theorem]{Corollary}
\theoremstyle{definition}
\newtheorem{assumption}{Assumption}
\newtheorem{remark}[theorem]{Remark}

\newcommand{\R}{\mathbb R}
\newcommand{\C}{\mathbb C}
\newcommand{\E}{\mathbb E}
\newcommand{\Prob}{\mathbb P}

\newcommand{\tr}{\operatorname{tr}}

\newcommand{\Law}{\operatorname{Law}}
\newcommand{\rank}{\operatorname{rank}}

\newcommand{\supp}{\operatorname{supp}}
\newcommand{\Id}{I}
\newcommand{\weak}{\Rightarrow}
\newcommand{\as}{\xrightarrow{\mathrm{a.s.}}}
\newcommand{\Imag}{\operatorname{Im}}
\newcommand{\Real}{\operatorname{Re}}
\newcommand{\op}{\mathrm{op}}
\newcommand{\HS}{\mathrm{HS}}

\title{Limiting ESD for a weighted sample covariance matrix.}
\author{Sandrine P\'ech\'e} \address{Universit\'e Paris Diderot, LPSM, 5 Rue Thomas Mann, 75013 Paris, France}
\email{peche@lpsm.paris}
\date{}

\begin{document}

\begin{abstract}
We study the empirical spectral distribution (ESD) of the random matrix 
$M_N=\frac1N\sum_{i=1}^N f(|\pi Y_i|^2)\,Y_iY_i^\top$,
where $\pi$ is a rank $k$ orthogonal projector on $\R^p$, $Y_1,\dots,Y_N$ are i.i.d. standard Gaussian vectors, and $k\to \infty$, $p/N\to \gamma$. We here assume that $f(x)=g(x/k)$ for some function $g$ discontinuous at $x=1$. We show that the asymptotic ESD is characterized by a  Marchenko--Pastur equation, independently of $k$. We also show that the result extends to sub Gaussian random vectors $Y_i$ and to the case where the projector is replaced with a more general weight $\phi(Y_i)$. Last, in the Gaussian case, we study the emergence of outliers  due to the weighting.  
\end{abstract}
\maketitle
\section{Model and results}
The study of sample covariance matrices is now well developed. Starting from the seminal result of \cite{MarcenkoPastur}, there has been a huge literature on the study of the empirical spectral distribution (ESD) of sample covariance matrices (see \cite{Silverstein95}). Since then, enormous progress has been accomplished. Local laws have been proved starting with \cite{KnowlesYau} for standard sample covariance matrices and  \cite{KnowlesYin} for non white sample covariance matrices (see also e.g. \cite{AltErdos} in the case of a variance profile, \cite{Paquette}  in the case of non separable sample covariance matrices) and the list is non exhaustive. Local laws have several important consequences. Based on local laws, one can deduce the delocalization of eigenvectors and prove the universality of local eigenvalue statistics. This is thus the most complete description of the spectral properties of sample covariance matrices, which are now well understood.  Most of the recent literature tries to remove the assumption that the entries of the sample covariance matrix $\frac{1}{N} XX^\top$ are i.i.d. random variables. In particular, it is of interest to introduce some dependency among the entries. 
There are multiple ways to do so ( see e.g. \cite{YinKrishnaiah}, \cite{BaiZhou}, \cite{Paul}). Robust estimators of sample covariance matrices may fall into this class of random matrices and are motivated by mathematical statistics questions. We refer the reader to \cite{Couillet1}, \cite{Couillet2}, \cite{Couillet3} for an example of such matrices in the context of random matrix theory. Therein, the authors study matrices of the form $\Sigma=1/N \sum_{i=1}^N u(x_i^t \Sigma^{-1} x_i)x_ix_i^t $ and obtain the limiting ESD for such Maronna's estimators. The function $u$ introduces some strong dependency on the entries here. 
In this article,  we study a related model, namely weighted sample covariance matrices of the form
$$M_N=\frac{1}{N}\sum_{i=1}^N f(|\pi Y_i|^2)Y_iY_i^t$$ for some  vectors $Y_i$ with i.i.d. coordinates and where $\pi$ is a rank $k$ projector on some subspace. Such matrices may be of interest in mathematical statistics (see e.g. \cite{Li}): the impact of a vector is weighted according to the projection on some fixed subspace $V.$ The weight is then generalized to be some function of the vector $Y_i$ only in the sequel of the article. We may refer to \cite{Engelen} where some similar procedure is used in the context of robust estimation of covariance. What is then the resulting spectrum of the associated sample covariance matrix? Of course, the dependency induced by the function $f$ prevents from using directly the established results for standard sample covariance matrices. One natural question is to question the role of the dimension $k$ (in particular in the large $k$ limit) as well as the geometry of the subspace $V$. We here propose a first study of such random matrix ensembles.

In the article, we first consider the following Gaussian model:
\begin{assumption}\label{ass:gauss}
\begin{enumerate}
\item $Y_1,\dots,Y_N\stackrel{iid}\sim \mathcal N(0,I_p)$.
\item $\pi$ is the orthogonal projector on the first  $k$ coordinates of  $\R^p$. Here, by the rotational invariance of the Gaussian distribution, every rank $k$ projector can be reduced to this case by an orthonormal change of basis. 
\item $N,p\to\infty$ with $p/N\to \gamma \in(0,\infty)$ and $k\to \infty$.
\item $g:\R_+\to S$ with $S$ a compact subset of $\R_+^*$; $g$ continuous on  $\R_+\setminus\{1\}$, with left and right limits at $x=1$:
$$g(1^-)=a,\qquad g(1^+)=b,\qquad a\ne b,\qquad a,b> 0.$$
We now set $f=f_N(x):=g(x/k)$ and $w_i:=f(|\pi Y_i|^2)=g(\zeta_i/k)$, $\zeta_i:=|\pi Y_i|^2$.
\end{enumerate}
\end{assumption}
We consider the weighted sample covariance matrix $$M_N=\frac1N\sum_{i=1}^N g(|\pi Y_i|^2/k)\,Y_iY_i^\top.$$

We prove the following result: let $\lambda_1\geq \lambda_2\geq \cdots\geq \lambda_p$ be the eigenvalues of $M_N$. We denote its spectral measure by $\mu_{M_N}=\frac{1}{p}\sum_{i=1}^p \delta_{\lambda_i}.$
\begin{theorem}\label{thm:gauss}
Assume that Assumption~\ref{ass:gauss} holds true, and let $H:=\frac12\delta_a+\frac12\delta_b$. Almost surely, $\mu_{M_N}\Rightarrow\mu$, whose Stieltjes transform $m(z)=\int\frac{d\mu(x)}{x-z}$ satisfies, for $z\in\C^+$,
\begin{equation}\label{eq:main}
m(z)=-\Big(z-\int\frac{t\,dH(t)}{1+\gamma\,t\,m(z)}\Big)^{-1}.
\end{equation}
\end{theorem}

\begin{remark}
It is a notable fact that $\alpha=\lim k/N$ does not play a role in the equation: the rank of the projector $\pi$ impacts the limiting empirical distribution only through the mechanism defining the limiting distribution $H$ (see Lemma~\ref{lem:weight}). Thus $m(z)$ does not depend on $k$ when $\gamma$, $a$, $b$ are fixed. 
\end{remark}
\begin{remark}
The columns $\sqrt{w_i}\,Y_i$ are i.i.d., so $M_N$ is a sample covariance matrix with independent columns. However, the concentration condition of \cite{BaiZhou} Theorem 1.1 on quadratic forms fails here, since $\operatorname{Var}(w_iY_i^\top BY_i)\asymp\operatorname{Var}(w)\,(\operatorname{tr}B)^2$ may be of order $p^2$ rather than $o(p^2)$ for some deterministic matrix $B$ ($B$ is the identity matrix e.g.). Correspondingly, the limit is the Marchenko–Pastur law with weight distribution $H$, not the one with population covariance $\mathbb \E[w] Id$.
\end{remark}

We also study the case where $k$ is fixed. In this case, the limiting ESD is well-known.

Define
\[
\nu_k:=\Law\left(g\left(\frac{\chi_k^2}{k}\right)\right).
\]

\begin{proposition}
    
\label{thm:kfixe}
Assume that $k$ is fixed, $p/N\to \gamma \in(0,\infty)$, and that
$g:\R_+\to S$, with $S$ a compact subset of $\R_+^*$; $g$ is Borel measurable. Then, almost surely,
\[
\mu_{M_N}\weak\mu_{\gamma,\nu_k},
\]
where $\mu_{\gamma,\nu_k}$ is the generalized Marchenko--Pastur law with
aspect ratio $\gamma$.
The Stieltjes transform $m$ of $\mu_{\gamma,\nu_k}$ is the unique solution in $\C^+$ of
$$m(z)=-\left (z-\int \frac{t}{1+\gamma tm(z)} d\nu_k(t)\right )^{-1}.$$
\end{proposition}
In other words, the bulk of the spectrum is not impacted by the dependence induced by the weighting. \\
In this setting, the projection $\pi$, or the weighting, may result into the existence of outliers. Note that the behavior of extreme eigenvalues in the context of dependent entries has already been studied in \cite{Chafai} or \cite{Basrak} e.g. The emergence of outliers is a new phenomenon to our knowledge. We study the possible locations of such outliers in the next proposition. \\
We write $Y_i=\begin{pmatrix}X_i \cr Z_i\end{pmatrix}$ where $X_i \in \R^k$ and $Z_i\in \R^{p-k}. $ We also set $W_N=\text{diag}(w_1, \ldots, w_N), Z=[Z_1, \ldots, Z_N]$
\begin{proposition}\label{prop: outlier}
Every possible limiting outlier location
$z\notin \text{supp}(\mu_{\gamma, \nu_k})\cup\{0\}$  must satisfy
\begin{equation} \label{eq: outliers}
z=\frac{1}{k}\E\left[
\frac{\chi_k^2g(\chi_k^2/k)}
{1+\gamma m(z)g(\chi_k^2/k)}
\right].
\end{equation}
\end{proposition}
For fixed $k$, as will be shown by the proof, the weighting alone may create exactly $k$ outliers, whose location is characterized. This is different from the emergence of outliers in the \cite{BBP} phase transition. \\
The convergence of the outliers requires a real solution to the equation (\ref{eq: outliers}). We can show that the bulk has no
eigenvalues in the chosen spectral gap so that an eigenvalue indeed separates when a real solution exists. 
We discuss here the existence of real solutions. \\
Set $s=\frac{\chi^2(k)}{k}$ with variance $2/k.$ Set also $u=\frac{m(z) g(s)}{1+\gamma m(z)g(s)}.$
The equation for spikes can be rewritten, using the Marchenko-Pastur equation, as 
$$\mathbb{E} ((s-1) u)=-1.$$
Outside the support of $\mu_{\gamma , \nu_k}$, $m(z)$ is real and one has that $z\mapsto m(z)$ is strictly increasing. Deriving the fixed point equation $z(m)=-1/m+\mathbb{E}\frac{g(s)}{1+\gamma m(z)g(s)}$ and writing that $\frac{dz}{dm}>0$ one obtains that $\mathbb{E}(u^2)<\gamma^{-1}.$
By Cauchy-Schwarz, $\left (\mathbb{E}(u(s-1))\right )^2\leq \text{Var}(s) \mathbb E u^2\leq 2/(k\gamma).$
This implies that $k\gamma<2$ is necessary so that a real solution exists. As a consequence there are no spikes in the case where $k\gamma\geq 2.$ \\
More generally, if $\lim_{z \to x_+}\rho (z)-z>0$ where $x_+$ is the right edge of the support of $\mu_{\gamma, \nu_k},$ this ensures the existence of outliers to the right of the support of $\mu_{\gamma, \nu_k}$. If $\lim_{z \to x_+}\rho (z)-z<0$, there is no outlier. 
We cannot handle the case where $\lim_{z \to x_+}\rho (z)-z=0.$

We give some simulation in the case where $k=1$, $\gamma=1/10$ and $g(x)=1_{x\leq 1}+3\cdot 1_{x>1}$ in Figure \ref{simuk=1}. In this case, one has that $k\gamma<2$. The spike is numerically computed around 3.3.  We compare the spectrum of  $M_N$ and that of $S_N=\frac{1}{N}\sum g(|\pi Y_i|^2/k) Y_i^{\mathbf g} (Y_i^{\mathbf g})^\top$ for some independent Gaussian vectors $Y_i^{\mathbf g}.$  We see that a spike may arise after a simulation over 50 (resp. 100) matrices of size $N=500$ $\gamma=1/10$ (resp. $N=4000$) on Figure \ref{simuk=1} (resp. Figure \ref{simuk=2}). 
\begin{figure}
\includegraphics[height=12cm,width=1\textwidth]{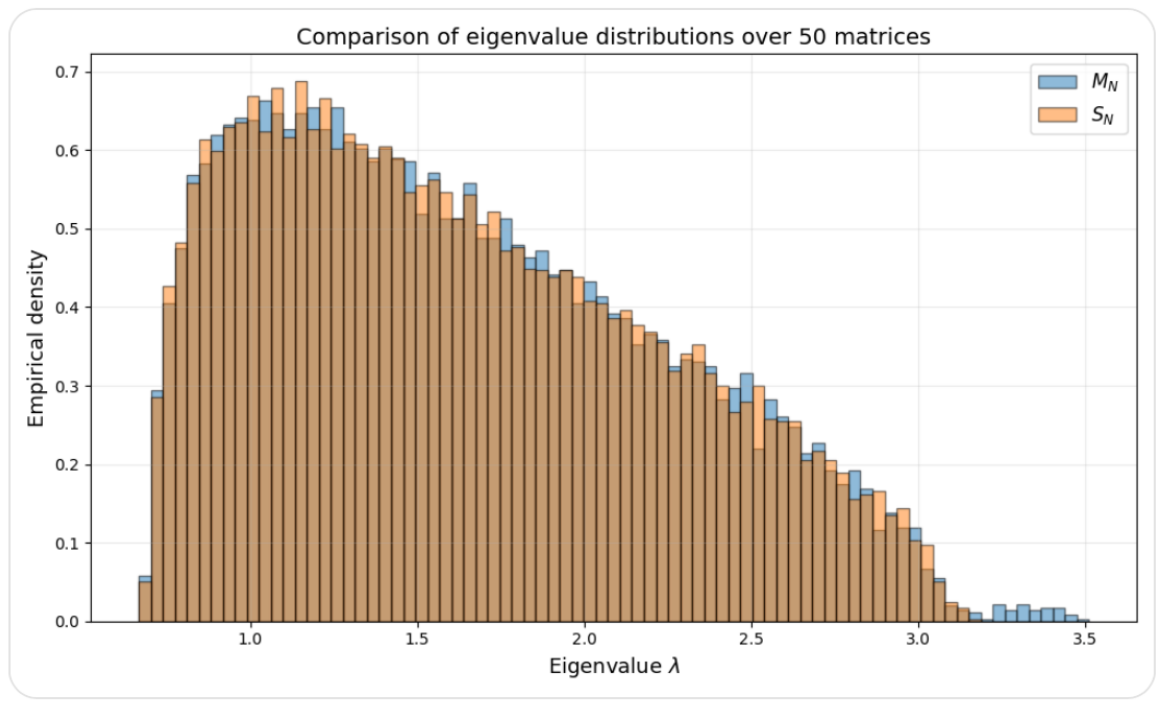}
\caption{A spike when $k=1, N=500$}
\label{simuk=1}
\end{figure}
\begin{figure}
\includegraphics[height=10cm,width=1\textwidth]{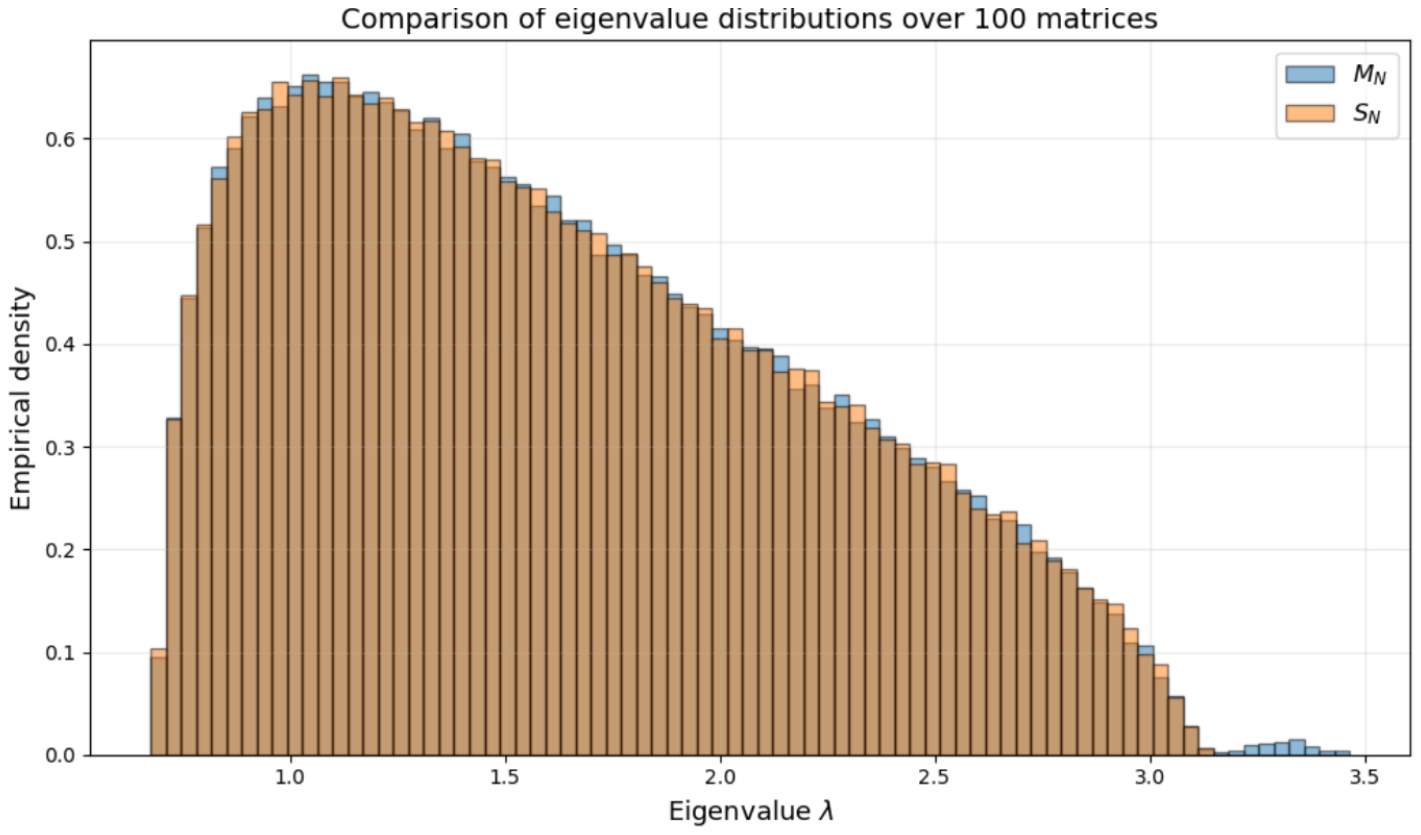}
\caption{A spike when $k=1, N=4000$}
\label{simuk=2}
\end{figure}
A simulation over 2000 matrices shows that the largest eigenvalue of $M_N$ is located at 3.29 with a standard deviation of $0.18$. In comparison, the edge of $S_N$ is located at $z\sim 3.09$ with a calculated mean at $3$ for the same parameters ($g,\gamma$, $N=500$).
\begin{remark}The result extends to  vectors $Y_i$ with i.i.d. symmetrically distributed sub-Gaussian coordinates instead of Gaussian vectors (in the case where $\pi$ is the orthogonal projector on the first $k$ coordinates). The equation for the outlier is then modified as follows: 
$$z=\frac{1}{k}\E\left[
\frac{\zeta_kg(\zeta_k/k)}
{1+\gamma m(z)g(\zeta_k/k)}
\right],$$ where $\zeta_k=\sum_{j\leq k}Y_{1j}^2$. This suggests that the possible location of spikes is highly non universal. The necessary condition for outliers also becomes $k\gamma < \E (Y_{11}^4)-1.$ 
\end{remark}

\vspace*{1cm}Last, we extend Theorem \ref{thm:gauss} to a universality statement. We show that Theorem \ref{thm:gauss}
remains valid if the $Y_i$'s are  sub- Gaussian i.i.d. centered vectors. Furthermore, the needed assumption on the weights is their convergence in distribution.

Consider a weighted sample covariance matrix $M_N=\frac{1}{N}\sum_{i=1}^N \phi(Y_i) Y_iY_i^\top,$ where we  assume that
\begin{assumption}\label{ass:uni}
\begin{enumerate}
\item $Y_1,\dots,Y_N\in\R^p$ are i.i.d.\ random vectors with i.i.d. centered sub-Gaussian coordinates of variance $1$ 
\item $N,p\to\infty$ with $p/N\to \gamma\in(0,\infty)$.
%\item $g:\R\to\R_+$ is bounded, continuous on  $\R\setminus\{1\}$, with left and right limits at $x=1$:
%$$g(1^-)=a,\qquad g(1^+)=b,\qquad a\ne %b,\qquad a,b> 0.$$
\item the weights $w_i:=\phi_N(Y_i)\in [a_{min}, a_{max}]$, $a_{min}>0$ converge in distribution to some probability measure $H$ : $w_i\Longrightarrow H.$
\end{enumerate}
\end{assumption}
We will drop the dependence on $N$ in $\phi_N=\phi$ in the sequel.
\begin{theorem}\label{thm: uni}
Under Assumption \ref{ass:uni}, in probability, $\mu_{M_N}\Rightarrow\mu$, where the Stieltjes transform $m(z)=\int\frac{d\mu(x)}{x-z}$ satisfies (\ref{eq:main}).
\end{theorem}

\begin{remark} Theorem \ref{thm: uni} is in particular valid in the case where $w_i=g(\frac{|\pi Y_i|^2}{k})$ where $\pi$ is an orthogonal projector of rank $k\to \infty$ , $g$ is a bounded function satisfying item (4) in Assumption \ref{ass:gauss}   and provided the following Central Limit theorem holds true: set
 $\zeta_i:=Y_i^\top\pi Y_i$ (so $\E[\zeta_i]=\tr\pi=k$). Then it holds that
 \[
\frac{\zeta_i-k}{\sqrt k}\ \Longrightarrow\ \mathcal N(0,\sigma^2)\qquad\text{for some fixed }\sigma^2>0.
\]
The CLT is in particular valid if $\sigma^2>0$ and $\pi$ is the orthogonal projector on the first $k$ coordinates. In this case, $H=1/2\delta_a+1/2\delta_b$ as in the Gaussian case. 
\end{remark}

The rest of the paper is concerned with the proof of the above results. We first prove Theorem \ref{thm:gauss} for Gaussian vectors, later extending it to sub-Gaussian vectors in Theorem \ref{thm: uni}. The proof for the Gaussian case is interesting by itself and provides quantitative bounds. The universality statement will follow from the analysis of a resolvent (fixed point) equation and uses standard concentration arguments.

\section{Proof of Theorem \ref{thm:gauss}}
We handle the Gaussian case: in this setting (and in comparison with Section \ref{sec: uni}), we obtain quantitative estimates on the Levy distance of the e.s.d. of $M_N$ to that of a well chosen standard sample covariance matrix studied by \cite{Silverstein95}.\\
One can write, because the $Y_i$'s are Gaussian, $Y_i=(X_i,Z_i)\in\R^k\times\R^{p-k}$ and using polar decomposition,
$$X_i=R_iU_i,\qquad R_i:=|X_i|>0,\quad U_i\ \text{uniform on }\mathbb S^{k-1},\quad R_i\perp\!\!\!\perp U_i.$$
Thus, $w_i=g(R_i^2/k)$ is a function of  $R_i$ only.

\begin{lemma}\label{lem:weight}
One has that $R_i/\sqrt k\to1$ in probability, and $w_i \Longrightarrow 1/2\delta_a+1/2\delta_b$.  Moreover,
$$\mathrm{Var}(R_i)\to\frac12\qquad(k\to\infty).$$
\end{lemma}

\begin{proof}
One has that
$R_i^2\sim\chi^2_k$, with $\mathbb{E}(R_i^2)=k$ and $\text{Var}(R_i^2)=2k$. By Hanson-Wright inequality one has that :
$$\Prob(|R_i^2-k|>t)\le 2\exp\big(-c_0\min(\frac{t^2}{k},\,t)\big)$$ for some constant $c_0.$\\
One can check that $$\mathrm{Var}(R_i)=k -2\left(\frac{\Gamma((k+1)/2)}{\Gamma (k/2)} \right)^2.$$ From this, one deduces that $\mathrm{Var}(R_i)\to 1/2,$ 
which is a classical result for $\sqrt{\chi^2_k}$. The Central Limit Theorem applied to $R_i^2-k$ gives the symmetry of the limiting distribution implying that $\Prob(R_i^2<k)\to\frac12$.
\end{proof}

\begin{remark} This is a key estimate: the radius $R_i$ has $O(1)$ fluctuations around $\sqrt k$, thus not growing with $k$. This is an essential point in the next comparison.
\end{remark}

We now build a matrix with the same limiting empirical distribution.

Let $(R_i')_{i=1}^N$ be i.i.d. copies with the same distribution than $R_i$ (i.e.\ $\sqrt{\chi^2_k}$), independent of $(R_i,U_i,Z_i)_i$. 

We now define $$\tilde X_i:=R_i'U_i,\qquad \tilde Y_i:=(\tilde X_i ,Z_i),\qquad M_N'':=\frac1N\sum_i w_i\tilde Y_i\tilde Y_i^\top.$$

Then it is straightforward that $$w_i\perp\!\!\!\perp\tilde Y_i$$ and $$\tilde Y_i\sim\mathcal N(0,I_p).$$ 

Observe that $M_N''=\frac1N\sum_iw_i\tilde Y_i\tilde Y_i^\top$ falls exactly in the class of Silverstein's sample covariance random matrices (studied by \cite{Silverstein95}) with  i.i.d. weights of distribution tending to \ $H$, independent of the Gaussian vectors $\tilde Y_i$.

The comparison of limiting empirical spectral distributions will now follow from the following Lemma.
Let $F^S$ denote the empirical spectral distribution of a matrix $S.$
Let also $L(\cdot, \cdot)$ denote the Levy distance defined by 
$$L(F^S, F^T)=\inf\{\epsilon, F^S(x-\epsilon)-\epsilon\leq F^T(x)\leq F^S(x+\epsilon)+\epsilon, \forall x\}.$$
\begin{lemma}[\cite{BaiSilverstein} Corollary A.42  ]\label{lem:bai}
For matrices $A,B$ of size $p\times n$, and $S=AA^\top$, $S'=BB^\top$, one has that
$$L^4(F^S,F^{S'})\le\frac2{p^2}\tr(AA^\top+BB^\top)\cdot\tr[(A-B)(A-B)^\top].$$
\end{lemma}
Using this lemma, one deduces the next Proposition
\begin{proposition}\label{prop:levy}One has that
$L(\mu_{M_N},\mu_{M_N''})\to0$ almost surely. 
\end{proposition}

\begin{proof}
Define the $p\times N$ matrices $A:=\frac1{\sqrt N}[\sqrt{w_1}Y_1|\cdots|\sqrt{w_N}Y_N]$, and similarly $B:=\frac1{\sqrt N}[\sqrt{w_1}\tilde Y_1|\cdots|\sqrt{w_N}\tilde Y_N]$ , so that $AA^\top=M_N$, $BB^\top=M_N''$.

One first has that $\tr(AA^\top)=\frac1N\sum_iw_i|Y_i|^2$. As $w_i$ is bounded and $|Y_i|^2=O_P(p)$,one has that $\tr(AA^\top)=O_P(p)$ ; and similarly $\tr(BB^\top)=O_P(p)$. This can be extended to 
$\tr(AA^\top)=O(p)$ a.s. by Hanson-Wright inequality. 
We now turn to the difference term $A-B=\frac1{\sqrt N}[\sqrt{w_i}(Y_i-\tilde Y_i)]_i$, and $Y_i-\tilde Y_i=((R_i-R_i')U_i,0)$, so that $|Y_i-\tilde Y_i|^2=(R_i-R_i')^2$. Thus one has that
$$\tr[(A-B)(A-B)^\top]=\frac1N\sum_iw_i(R_i-R_i')^2.$$
By Lemma~\ref{lem:weight}, $R_i,R_i'$ are i.i.d. of variance tending to $\frac12$, independent of each other. We deduce that $\E[(R_i-R_i')^2]=\mathrm{Var}(R_i)+\mathrm{Var}(R_i')\to1$ — which is bounded, and does not grow with $k$. One then has that
$$\E\,\tr[(A-B)(A-B)^\top]=O(1).$$

%Using now Lemma~\ref{lem:bai}, one has that
%$$\E\,L^4(\mu_{M_N},\mu_{M_N''})\longrightarrow0.$$
The almost sure convergence follows from the upper bound
\begin{eqnarray}&&\mathbb{E}\Big | \tr[(A-B)(A-B)^\top]-\E\tr[(A-B)(A-B)^\top]\Big |^4= \cr &&\E \Big |\frac1N\sum_iw_i(R_i-R_i')^2-\E \frac1N\sum_iw_i(R_i-R_i')^2\Big |^4.\end{eqnarray}
Set $V_i=w_i(R_i-R_i')^2-\E w_i(R_i-R_i')^2.$
Following \cite{LiBaiHu}, Lemma 9, we use Rosenthal's inequality, which states that 
$$\E \Big| \sum_{i=1}^N V_i\Big |^r\leq C_r \left( \sum_{i=1}^N \E |V_i|^r +(\sum_{i=1}^N \E V_i^2)^{r/2}\right )$$ for some constant $C_r.$
We now use the fact that $$
\sup_{N,k}\mathbb E|R_i-R_i'|^8<\infty.$$
From this, one can show that 
$$\mathbb{E}\Big | \tr[(A-B)(A-B)^\top]-\E\tr[(A-B)(A-B)^\top]\Big |^4=O(N^{-2}),$$ which is  summable in $N$. The conclusion follows from Borel-Cantelli's lemma. One has that $\tr[(A-B)(A-B)^\top-\mathbb{E} \tr[(A-B)(A-B)^\top$ converges a.s. towards $0,$
 which is then enough to ensure the result.
\end{proof}

There now remains to give the limiting ESD of $M_N''$, which is well known in random matrix theory since \cite{Silverstein95}. 
\begin{proposition}\label{prop:classical}
Almost surely, one has that $\mu_{M_N''}\Rightarrow\mu$, where $m(z)$ satisfies \eqref{eq:main}.
\end{proposition}

As $L(\mu_{M_N},\mu_{M_N''})\to 0$ a.s., $\mu_{M_N}$ and $\mu_{M_N''}$ have the same weak limit. This common limit admits $m(z)$, the solution of \eqref{eq:main}, as a Stieltjes transform.

\section{The case where $k$ is fixed}
 We first examine the limiting e.s.d. of $M_N$ and prove Proposition \ref{thm:kfixe}. 
\begin{proof} We write $Y_i=\begin{pmatrix}X_i \cr Z_i\end{pmatrix}$ where $X_i \in \R^k$ and $Z_i\in \R^{p-k}. $ We then define the matrices
\[
X=(X_1,\ldots,X_N),\qquad Z=(Z_1,\ldots,Z_N), \quad W_N=\text{diag}(w_1, \ldots, w_N).
\]
Then one has that 
\[
M_N=\frac1N
\begin{pmatrix}
XW_NX^\top&XW_NZ^\top\\
ZW_NX^\top&ZW_NZ^\top
\end{pmatrix}.
\]
Define
\[
S_N=\frac1NZW_NZ^\top,\qquad
\widehat M_N=
\begin{pmatrix}0_{k\times k}&0\\0&S_N\end{pmatrix}.
\]
The matrix $S_N$ then falls in the class studied by \cite{Silverstein95}. Its limiting spectral measure is then well understood.  
Now, because
$\rank(M_N-\widehat M_N)\leq2k,$
one has that 
$\sup_{x\in\R}|F^{M_N}(x)-F^{\widehat M_N}(x)|
\leq\frac{2k}{p}.$
Proposition \ref{thm:kfixe} then follows.
\end{proof}

 We now investigate the possible location of outliers and show Proposition \ref{prop: outlier}. This is the object of the rest of this Section.
\begin{proof}
Write
\[
M_N=\begin{pmatrix}A_N&B_N\\B_N^\top&S_N\end{pmatrix},
\]
where
\[
A_N=\frac1NXW_NX^\top,\qquad
B_N=\frac1NXW_NZ^\top.
\]

\begin{lemma}\label{prop:schur}
If $z\notin\operatorname{Spec}(S_N)$, then
\[
z\in\operatorname{Spec}(M_N)
\quad\Longleftrightarrow\quad
\det K_N(z)=0,
\]
where
\[
K_N(z)=A_N-z\Id_k-B_N(S_N-z\Id)^{-1}B_N^\top.
\]

\end{lemma}

\begin{proof}
This is a straightforward consequence of the Schur-complement formula.
\end{proof}

Define now
\[
V_N=\frac1{\sqrt N}ZW_N^{1/2},\qquad
\widetilde S_N=V_N^\top V_N
=\frac1NW_N^{1/2}Z^\top ZW_N^{1/2},
\]
and
\[
T_N(z)=(\widetilde S_N-z\Id_N)^{-1}.
\]

\begin{proposition}\label{prop:companion}
For $z\neq0$ outside $\operatorname{Spec}(S_N)$,
\[
K_N(z)
=-z\left[
\Id_k+\frac1NXW_N^{1/2}T_N(z)W_N^{1/2}X^\top
\right].
\]
\end{proposition}

\begin{proof}
The identity follows from
\[
\Id-V^\top(VV^\top-z\Id)^{-1}V
=-z(V^\top V-z\Id)^{-1}.
\]
\end{proof}

We are now in position to derive a deterministic equation for candidate outliers.

For complex $z\notin\supp(\mu)$ define
\[
\rho(z)=\frac{1}{k}\E\left[
\frac{\chi_k^2g(\chi_k^2/k)}
{1+\gamma m(z)g(\chi_k^2/k)}
\right].
\]
We prove the following.
\begin{proposition}\label{prop:limitK}
Uniformly on compact subsets $K$ of
$\R\setminus(\supp(\mu_{\gamma, \nu_k})\cup\{0\})$, one has that  
\[
K_N(z)\as(\rho(z)-z)\Id_k.
\]
\end{proposition}
\begin{proof}
We prove Proposition \ref{prop:limitK}.
The proof will be based on an
anisotropic
deterministic equivalent for $T_N$ conditional to $X$ and $W_N$.\\ 

We set 
$$
m_N(z)=-\Big(z-\int \frac{t\,d\nu_N(t)}{1+\gamma_N\,t\,m_N(z)}\Big)^{-1},
\qquad \nu_N=\frac1N\sum_{i=1}^N\delta_{w_i},\quad \gamma_N=\frac pN .
$$
Denote by $\rho_N$ the probability measure with Stieltjes transform $m_N.$

We first show the following Lemma. 

\begin{lemma}\label{lem: supprho}
Let $K$ be a compact subset of 
$\C\setminus(\supp(\mu_{\gamma, \nu_k})\cup\{0\})$. One has that
$K\cap \text{supp}(\rho_N)=\emptyset$ for $N$ large enough and $\sup_{z \in K}|m_N(z)-m(z)|\to 0$ as $N\to \infty.$
\end{lemma}
\begin{proof}
We show Lemma \ref{lem: supprho}. Let $\gamma'>0,$ $\nu'$ be a probability distribution with compact support $S'\subset (0, \infty).$ Define the Marchenko-Pastur transformation:
$$z_{\gamma', \nu'}(m)=-\frac{1}{m}+\int \frac{t d\nu'(t)}{1+\gamma' tm}=:-\frac{1}{m}+\int \phi_{\gamma'}(t,m)d\nu'(t).$$
Let $\mu_{\gamma', \nu'}$ be the associated Marchenko-Pastur distribution. Then a real point $x\not=0$ lies outside the support $\text{supp}(\mu_{\gamma', \nu'})$ iff there exists $m\in \R\setminus\{0\}$ such that 
$$1+\gamma' tm\not=0 \:\forall t\in S'; z_{\gamma', \nu'}(m)=x; \: z_{\gamma', \nu'}'(m)>0.$$
In that case, one has that $m=m_{\gamma', \nu'}(x).$\\

Let $S=\text{supp}(\nu_k).$
The $w_i$ are i.i.d. with law $\nu_k$, $\nu_N\Longrightarrow \nu_k$, and $\gamma_N\to \gamma$ and one has that $\text{supp}(\nu_N)\subset S.$ We work on this event. We also write 
$z=z_{\gamma, \nu_k};\: z_N=z_{\gamma_N, \nu_N}.$ Note that it is enough to consider $K_{\R}=K\cap \R$ as $\text{supp}(\rho_N)$ is real.\\
Let now $x\in K_{\R}:$ there exists $m(x)\not=0$ such that $1+\gamma t m(x)\not=0$ $\forall t \in S$ and $z'(m(x))>0.$ The map $m$ is continuous so $\Gamma=m(K_{\R})$ is a compact subset of $\R\setminus \{0\}.$ By compactness of $\Gamma \times S$, there exists $\delta_0, c, \epsilon >0$ such that on $\Gamma_{\delta_0}=\{m, \text{dist}(m, \Gamma)\leq \delta_0\}$ and for $|\gamma'-\gamma|\leq \epsilon$ one has that 
$$|m|\geq \delta_0;\: |1+\gamma' tm|\geq \delta_0, \: t\in S \text{ and }z_{\gamma, \nu_k}'(m)\geq c.$$
On $S\times \Gamma_{\delta_0}\times[\gamma-\epsilon, \gamma+\epsilon],$ the functions 
$t\mapsto \phi_{\gamma'}(t,m)$ and $t\mapsto \frac{t^2}{(1+\gamma' tm)^2}$ are uniformly continuous and bounded. Now, one has that 
$$z_N(m)-z(m)=\int (\phi_{\gamma_N}-\phi_{\gamma})d\nu_N(t)+\left ( \int \phi_{\gamma}d\nu_N-\int \phi_{\gamma}d\nu_k \right ).$$
The first term is upper bounded by $\sup_{S\times \Gamma_{\delta_0}}|\phi_{\gamma_N}-\phi_{\gamma}|\to 0$ as $N\to \infty.$
For the second term, the family $\{\phi_{\gamma}(\cdot, m)\}$ is a compact family in the set of continuous functions on $S$. So the weak convergence $\nu_N\Longrightarrow \nu_k$ holds uniformly for $\phi_{\gamma}$ in $m$. Applying the same to $z'_N, z'$ we deduce that
$$\sup_{\Gamma_{\delta_0}}\left ( |z-z_N|+|z'-z'_N|\right )\to 0 \text{ as }N \to \infty.$$
Fix now $\delta \in (0, \delta_0], x\in K_{\R}.$ Since $z'\geq c$ on $\Gamma_{\delta_0}$, we deduce that $z(m(x)-\delta)\leq x-c\delta$ and $z(m(x)+\delta)\geq x+c\delta.$ Let now $N_0(\delta)$ be such that for $N\geq N_0$
$$\sup_{\Gamma_{\delta_0}}|z-z_N|< c\delta;\: \sup_{\Gamma_{\delta_0}}|z'-z_N'|<c/2;\:|\gamma-\gamma_N|<\epsilon. $$
Then $z_N$ is continuous on the interval and $z_N(m(x)+\delta)>x>z_N(m(x)-\delta).$
Thus there exists $s_N \in [m(x)-\delta, m(x)+\delta]$ such that $$z_N(s_N)=x.$$
In addition, one has that 
$z_N'(s_N)\geq c/2>0$, $s_N\not=0$ and $1+\gamma_N ts_N\not=0, t\in S$.
From the latter estimates, one deduces that $x\notin \text{supp}(\rho_N)$ and $m_N(x)=s_N.$ Furthermore $|m_N(x)-m(x)|\leq \delta.$ Note that $N_0$ does not depend on $x$ so this gives the first statement of the lemma and shows that $\sup_{K_{\R}}|m_N-m|\to 0$ as $N\to \infty.$\\
The extension to complex points essentially follows from the same argument. We choose $r>0$ so that $K''=\{z \in \C, \text{dist}(z, K)\leq 2r\}$ avoids $\text{supp}(\mu_{\gamma, \nu_k})\cup \{0\}$ and set $K'=\{z \in \C, \text{dist}(z, K)\leq r\}.$ We use the same arguments as above for $K''$. The $m_N$ are analytic on a neighborhood of $K'$ and for $N$ large enough, there are bounded there by $2/r.$ Using the a.s. convergence of $\rho_N \Longrightarrow \mu_{\gamma, \nu_k}$, we deduce that $m_N(z)$ locally converges uniformly to $m(z).$ This gives the result.

\end{proof}
 
We now use the anisotropic local law proved by \cite{KnowlesYin}, Theorem 3.16. From now on, we  consider the conditional distribution with respect to $X$ (and $W_N$). This establishes that for all unit vectors $u,v$ and uniformly over $z\in U(\epsilon):=\{z \in \mathbb{R}, |z|\leq \epsilon^{-1}, |z|\geq \epsilon; d(z, supp(\rho_N))\geq \epsilon\}$, for $N$ large enough
$$u^* \left[
T_N(z)+\frac1z(\Id_N+\gamma m_N(z)W_N)^{-1}
\right]v\prec N^{-1/4}.$$ 
Hereabove the term $X\prec Y$ uniformly over $U$ means that, for any fixed $\epsilon, D>0$
$$\sup_{u \in U}\mathbb{P} \left (|X(u)|>N^{\epsilon}Y(u)\right )<N^{-D}.$$
Note that the above bound is not optimal but is enough for our purposes.
Here, to extend the validity of the above estimate from $z=E+i\eta $ with $\eta\geq N^{-1/2+\epsilon}$ to $z=E$, we have used that with probability $1$, $\tilde S_N$ has no eigenvalue in $K_r$ if $K_r=\{x, \text{dist}(x, K)\leq r\}$ and $K_{2r}$ avoids $\text{supp}(\mu_{\gamma, \nu_k})\cup \{0\}.$ This follows from Theorem 1.1. in \cite{BaiSilverstein2}. We skip the detail.

\begin{lemma}\label{ass:iso}
Let $K$ be a compact subset
$K\Subset\R\setminus(\supp(\mu_{\gamma, \nu_k})\cup \{0\})$. There exists $\epsilon=\epsilon (K)$ such that a.s. $K \subset U(\epsilon)$ for $N$ large enough and
\[
\sup_{z\in K}\left\|
\frac1NXW_N^{1/2}
\left[
T_N(z)+\frac1z(\Id_N+\gamma m_N(z)W_N)^{-1}
\right]
W_N^{1/2}X^\top
\right\|\as0.
\]
\end{lemma}
\begin{proof}
We apply the local law to the vectors $u_{r,N}=W_N^{1/2}X^\top e_r, r=1, \ldots,k.$ It is an easy consequence of the strong law of large number that $|u_{r,N}|=O(\sqrt N)$ almost surely. We can apply the anisotropic local law to the $k$ vectors simultaneously, yielding Lemma \ref{ass:iso}.
\end{proof}

By Proposition \ref{prop:companion} and Lemma \ref{ass:iso}, we deduce that
\[
\frac1NXW_N^{1/2}T_N(z)W_N^{1/2}X^\top
=-\frac1{zN}\sum_{i=1}^N
\frac{w_iX_iX_i^\top}{1+\gamma m(z)w_i}+o(1).
\]
This follows from the fact that $\sup_{z \in K }|m_N(z)-m(z)|\to 0$ a.s. Here we have used Lemma \ref{lem: supprho}.

Since $k$ is fixed, the strong law of large numbers applies entrywise. Rotational
invariance of the Gaussian law ($w(X)$ is a function of $|X|$ only) implies
\[
\E\left[\frac{w(X)XX^\top}{1+\gamma m(z)w(X)}\right]
=\rho(z)\Id_k.
\]
Substituting the latter into the preceding formula proves the result.
As a consequence we have shown Proposition \ref{prop: outlier}.
\end{proof}

We finish with the sufficient condition for outliers. 
For $z\in\mathbb R\setminus(\operatorname{supp}\mu_{\gamma,\nu_k}\cup\{0\})$ we set
\[
h(z):=\rho(z)-z.
\]
For a closed interval $I\subset\mathbb R$ and $r>0$, let
$I_r:=\{x\in\mathbb R:\ \operatorname{dist}(x,I)\le r\}$.

\begin{proposition}\label{prop:suffoutliers}
Let $I$ be a compact interval with
$I\subset\mathbb R\setminus(\operatorname{supp}(\mu_{\gamma,\nu_k})\cup\{0\})$.
Then $h$ is continuous and strictly decreasing on $I$, with
\begin{equation}\label{eq:h-decrease}
h(z')\le h(z)-(z'-z)\qquad (z<z'\in I),
\end{equation}
so that $h$ has at most one zero in $I$. Moreover, almost surely, for $N$ large enough:
\begin{enumerate}
\item[(i)] if $h$ has no zero in $I$, then $M_N$ has no eigenvalue in $I$;
\item[(ii)] if $h(z_0)=0$ for some $z_0\in\operatorname{int}(I)$, then $M_N$ has exactly $k$
eigenvalues in $I$, counted with multiplicity, and all of them converge to $z_0$.
\end{enumerate}
\end{proposition}
In particular, a limiting outlier location outside
$\operatorname{supp}(\mu_{\gamma,\nu_k})\cup\{0\}$ exists if and only if the equation $h(z)=0$ has a real solution there.

\begin{proof}
Choose $r>0$ such that
$I_{2r}\subset\mathbb R\setminus(\operatorname{supp}\mu_{\gamma,\nu_k}\cup\{0\})$.

We first show that $S_N$ has no eigenvalue near $I.$ This follows from 
Lemma \ref{lem: supprho} applied to 
$I_{2r}$ and Theorem 1.1. in \cite{BaiSilverstein2}.
On this event $K_N(z)$
is a real symmetric $k\times k$ matrix, of class $C^1$ in $z\in I_r$.\\
One may also note that the multiplicity of $z$ as an eigenvalue of $M_N$ equals $\dim(\ker K_N(z))$.\\
We now show the monotonicity of $h$.
For $z\in I_r$, one has that
\[
K_N'(z)=-I_k-B_N(S_N-z)^{-2}B_N^\top\ \preceq\ -I_k .
\]
Therefore $K_N(z')\preceq K_N(z)-(z'-z)I_k$ for $z<z'$ in $I_r$.
Let $\lambda_1(z)\ge\dots\ge\lambda_k(z)$ denote the ordered eigenvalues of $K_N(z)$.
By the Courant--Fischer theorem, $\lambda_j(z')\le\lambda_j(z)-(z'-z)$, and by
Weyl's inequality each $\lambda_j$ is continuous. In particular, each $\lambda_j$ is
strictly decreasing.
Now one has that
$\sup_{z\in I_r}\|K_N(z)-h(z)I_k\|\to0$ almost surely.
Passing to the limit gives \eqref{eq:h-decrease} on $I_r$.
Hence $h$ is continuous, strictly decreasing, has at most one zero, and
$|h(z)|\ge|z-z_0|$ whenever  $h(z_0)=0 .$

We can now study case  $(i)$. If $h$ has no zero in $I$, then $c:=\min_I|h|>0$ by compactness.
For $N$ large, $\sup_I\|K_N-hI_k\|<c/2$, so $K_N(z)$ is invertible for every $z\in I$.
Thus $M_N$ has no eigenvalue in $I$.\\
We turn to case $(ii)$. 
Let $\delta>0$ be small enough that $[z_0-\delta,z_0+\delta]\subset\operatorname{int}(I)$.
For $z\in I\setminus(z_0-\delta,z_0+\delta)$, one has that $|h(z)|\ge\delta$. We can similarly conclude as above that 
$z$ is not an eigenvalue of $M_N$.

Inside the interval $[z_0-\delta,z_0+\delta]$, one has that $h(z_0-\delta)\ge\delta$ and $h(z_0+\delta)\le-\delta$.
By Weyl's inequality, one has that
\[
\max_{j\le k}\big|\lambda_j(z_0\pm\delta)-h(z_0\pm\delta)\big|
\le\big\|K_N(z_0\pm\delta)-h(z_0\pm\delta)I_k\big\|<\delta/2
\]
for $N$ large. Hence, for every $j\le k$,
$\lambda_j(z_0-\delta)>0>\lambda_j(z_0+\delta)$.
Each $\lambda_j$ is continuous and strictly decreasing, so it vanishes at exactly one point
of $(z_0-\delta,z_0+\delta)$. The number of eigenvalues of $M_N$ in the interval,
counted with multiplicity, is
\[
\sum_{z}\dim(\ker K_N(z))=\#\{(z,j):\ \lambda_j(z)=0\}=k .
\]
Taking $\delta=1/q$, $q\in\mathbb N$ large, and intersecting countably many almost sure
events, the $k$ eigenvalues converge to $z_0$.
\end{proof}
This finishes the proof of Proposition \ref{prop: outlier}. \end{proof}

\section{Universality \label{sec: uni}}

We here prove Theorem \ref{thm: uni}.
Throughout this section, $z=E+i\eta\in\C^+$ is fixed, and we set $R(z):=(M_N-z)^{-1}$, $m_N(z):=\frac1p\tr R(z)$.

The main formula will be the Sherman--Morrison formula. For $i\le N$ set $M_N^{(i)}:=M_N-\frac1Nw_iY_iY_i^\top$, $R^{(i)}(z):=(M_N^{(i)}-z)^{-1}$, and $\mathcal F^{(i)}:=\sigma(Y_j,\,j\ne i)$. Set $Q_i(z):=Y_i^\top R^{(i)}(z)Y_i$, one has that
\begin{eqnarray}
&&Y_i^\top R(z)Y_i=\frac{Q_i(z)}{1+\frac{w_i}NQ_i(z)}, \text{ and}\cr 
&&
z\,\tr R(z)=\tr(M_NR(z))-p=\frac1N\sum_{i=1}^N\frac{w_iQ_i(z)}{1+\frac{w_i}NQ_i(z)}-p.
\end{eqnarray}
Dividing by $p$, we obtain that :
\begin{equation}\label{eq:star}
z\,m_N(z)=\frac1{Np}\sum_{i=1}^N\frac{w_iQ_i(z)}{1+\frac{w_i}NQ_i(z)}-1.
\end{equation}
We are going to analyze this formula using concentration results. We start with some preliminary lemmas. Most of these lemmas are standard analysis of the resolvent formula. 

\subsection{Some preliminary lemmas}

We work on $\mathcal E_N:=\{\|M_N\|_\op\le C_0\}\cap\{\max_i\|Y_i\|^2\le3p/2\}\cap\{\min_i\|Y_i\|^2\ge p/2\}$, of probability tending to 1; all bounds below hold on $\mathcal E_N$, which suffices to conclude to the convergence in probability.

We now use the standard leave-one-out comparison formula. 
 By construction $M_N^{(i)}$ is $\mathcal F^{(i)}$-measurable, hence independent of $Y_i$; moreover $M_N^{(i)}\succeq0$ and  $\|R^{(i)}(z)\|_\op\le1/\eta$, $\|R^{(i)}(z)\|_\HS\le\sqrt p/\eta$ on $\mathcal E_N$, for all $i$.

\begin{lemma}\label{lem:1}
On $\mathcal E_N$, uniformly in $i\le N$, one has that
\begin{equation}\label{eq:L1}
\big|\tr R(z)-\tr R^{(i)}(z)\big|\le \frac{\|\phi\|_\infty}{N\eta^2}\cdot\frac{3p}2=O_\eta\Big(\frac pN\Big)=O_\eta(1).
\end{equation}
\end{lemma}

\begin{proof}
By the resolvent identity, $R(z)-R^{(i)}(z)=R(z)\big(M_N^{(i)}-M_N\big)R^{(i)}(z)=-\frac{w_i}N R(z)Y_iY_i^\top R^{(i)}(z)$. Taking traces one deduces that
\[
\tr R(z)-\tr R^{(i)}(z)=-\frac{w_i}{N}\,Y_i^\top R^{(i)}(z)R(z)Y_i.
\]
Thus one gets that 
\[
\big|Y_i^\top R^{(i)}(z)R(z)Y_i\big|\le\|R^{(i)}(z)\|_\op\|R(z)\|_\op\|Y_i\|^2\le\frac1{\eta^2}\|Y_i\|^2\le\frac{3p}{2\eta^2},
\] on $\mathcal E_N.$
\end{proof}

We now use Hanson-Wright's inequality (see Theorem 6.2.2. in \cite{Vershynin}) to investigate the Sherman-Morrison formula.
Recall that the $Y_i$'s are i.i.d. sub Gaussian random vectors with $||Y_{11}||_{\psi_2}\leq K.$
\begin{lemma}\label{lem:2}
There is $c_{HW}=c_{HW}(K)>0$ such that, for every $i\le N$ and $t>0$,
\begin{eqnarray}
&&\Prob\Big(\big|Q_i(z)-\tr R^{(i)}(z)\big|>t\ \Big|\ \mathcal F^{(i)}\Big)\cr&&\le4\exp\left(-c_{HW}\min\left(\frac{t^2}{\|R^{(i)}(z)\|_\HS^2},\ \frac{t}{\|R^{(i)}(z)\|_\op}\right)\right),
\end{eqnarray}
where $Q_i(z):=Y_i^\top R^{(i)}(z)Y_i$.
\end{lemma}

\begin{proof}This follows from the real Hanson--Wright inequality (see \cite{RudelsonVershynin2013}) applied to  $R^{(i)}(z)=A+iB$ with $A=\Real R^{(i)}(z)$, $B=\Imag R^{(i)}(z)$ (which are real symmetric and $\mathcal F^{(i)}$-measurable). \end{proof}

\begin{corollary}\label{cor:2}
Let $\kappa:=2/c_{HW}$ and $t_N:=\sqrt{\kappa\,p\log N}/\eta$. Then, on $\mathcal E_N$, with probability $\to1$,
\begin{equation}\label{eq:L2}
\max_{i\le N}\big|Q_i(z)-\tr R^{(i)}(z)\big|\le t_N=O\big(\sqrt{p\log N}\big)=o(p).
\end{equation}
\end{corollary}

\begin{proof}
On $\mathcal E_N$, $\|R^{(i)}(z)\|_\HS\le\sqrt p/\eta$ and $\|R^{(i)}(z)\|_\op\le1/\eta$, so
\[
\frac{t_N^2}{\|R^{(i)}(z)\|_\HS^2}\ge\kappa\log N,\qquad \frac{t_N}{\|R^{(i)}(z)\|_\op}\ge\sqrt{\kappa p\log N}.
\]
 Hence, by Lemma~\ref{lem:2} and the choice $c_{HW}\kappa=2$,
\[
\Prob\big(|Q_i(z)-\tr R^{(i)}(z)|>t_N\mid\mathcal F^{(i)}\big)\le4N^{-2}.
\]
This yields the announced claim.
\end{proof}

Combining \eqref{eq:L1} and \eqref{eq:L2}, with probability $\to1$, one has that
\begin{equation}\label{eq:L3}
\max_{i\le N}\big|Q_i(z)-p\,m_N(z)\big|=o(p),\qquad \max_{i\le N}\big|m_N^{(i)}(z)-m_N(z)\big|=O_\eta(1/N),
\end{equation}
where $m_N^{(i)}(z):=\tr R^{(i)}(z)/p$.

We now examine in more detail the Sherman-Morrison formula, showing that the denominators do not vanish.

\begin{lemma}[Uniform lower bound on denominators]\label{lem:4}
For $N$ large enough, there is $c_1>0$ such that
\[
\min_{i\le N}\big|1+w_i\gamma_N m_N^{(i)}(z)\big|\ \ge\ c_1,\qquad \min_{i\le N}\Big|1+\frac{w_i}NQ_i(z)\Big|\ \ge\ c_1,\qquad \gamma_N:=p/N.
\]
\end{lemma}

\begin{proof}
Let $m(z)$ be the Stieltjes transform of a probability distribution $\nu$. Then $\frac{z}{\lambda-z}=-1+\frac{\lambda}{\lambda-z}.$
As a consequence $\Imag (z m(z))=\int \lambda \eta \frac{1}{|\lambda-z|^2}d\nu(\lambda)\geq 0.$
Thus, for $w\geq 0$, one has that 
$\Imag(z(1 + w\gamma_N m(z))) = \eta + w\gamma_N \Imag(zm(z)) \geq \eta,$ which yields that $|1+w\gamma_Nm(z)|\geq \eta/|z|.$
We can thus apply this result to 
$m_N^{(i)}(z)$ which is the Stieltjes transform of the  spectral measure of $M_N^{(i)}$. 

For the second bound, one has that $Q_i(z)=Y_i^\top R^{(i)}(z)Y_i$ is of the same form as above. The same argument applies by using the spectral decomposition of $R^{(i)}.$
\end{proof}

\subsection{Approximation of the Sherman--Morrison fraction}

We can now proceed to the analysis of the Sherman-Morrison formula using Lemmas \ref{lem:4} and (\ref{eq:L3}).
\begin{lemma}\label{lem:5}
With probability $\to1$,
\begin{equation}\label{eq:L5}
\max_{i\le N}\left|\frac{w_iQ_i(z)}{1+\frac{w_i}NQ_i(z)}-\frac{w_i\,p\,m_N^{(i)}(z)}{1+w_i\gamma_Nm_N^{(i)}(z)}\right|=o(p).
\end{equation}
\end{lemma}

\begin{proof}
Write $a_i=w_iQ_i(z)$, $a_i'=w_i\,p\,m_N^{(i)}(z)$, $b_i=1+\frac{w_i}NQ_i(z)$, $b_i'=1+w_i\gamma_Nm_N^{(i)}(z)$. By \eqref{eq:L3}, one has that $|a_i-a_i'|=w_i|Q_i(z)-pm_N^{(i)}(z)|\le\|\phi \|_\infty\,o(p).$  By Lemma~\ref{lem:4}, $|b_i|,|b_i'|\ge c_1$. Also $|b_i-b_i'|=\frac{w_i}N|Q_i(z)-pm_N^{(i)}(z)|\le\frac{\|\phi\|_\infty}N\,o(p)=o(1)$, and $|a_i'|=w_ip|m_N^{(i)}(z)|\le\|\phi\|_\infty p/\eta$. We then deduce that
\[
\left|\frac{a_i}{b_i}-\frac{a_i'}{b_i'}\right|\le\frac{o(p)}{c_1}+\frac{(\|\phi\|_\infty p/\eta)\,o(1)}{c_1^2}=o(p),
\]
uniformly in $i\le N$.
\end{proof}

Set now $\psi_i:=\dfrac{w_i\,m_N^{(i)}(z)}{1+w_i\gamma_N\,m_N^{(i)}(z)}$. Combining \eqref{eq:star} and \eqref{eq:L5}:
\begin{equation}\label{eq:starstar}
z\,m_N(z)=\frac1N\sum_{i=1}^N\psi_i+o_P(1)-1.
\end{equation}
This formula (\ref{eq:starstar}) is the basis formula we are now going to investigate thanks to the following lemma. 

Note that the $\psi_i$ are identically distributed but they are not independent. We control their average by a bounded-differences inequality. 

\begin{lemma}\label{lem:6}
With probability $\to1$,
\begin{equation}\label{eq:L6}
\frac1N\sum_{i=1}^N\psi_i=\E\Big[\frac1N\sum_{i=1}^N\psi_i\Big]+O_P(N^{-1/2}).
\end{equation}
\end{lemma}

\begin{proof}
The map $(w,x)\mapsto h(w,x):=\frac{wx}{1+w\gamma_Nx}$ is, on the set $\{|1+w\gamma_Nx|\ge c_1\}$, Lipschitz in $x$ with constant $\le\|\phi\|_\infty/c_1^2$.
Replacing $Y_j$ by an independent copy $Y_j'$ for $i\ne j$ is a rank $2$ perturbation. The resulting change in $m_N^{(i)}(z)$   is $O(1/N)$ by the rank inequality. Indeed denote by $F_A$ be the empirical spectral distribution of a Hermitian matrix $A.$ Then, if $A$ and $B$ are both Hermitian with $\text{rank}(A-B)\leq r$  $\sup_x |F_A(x)-F_B(x)|\leq r/p$ .  Hence the resulting change in $\psi_i$ is $O(1/N)$. For $i=j$, $\psi_j$ changes by $O(1)$ due to $w_j$. McDiarmid's inequality then gives, for $t>0$,
\begin{eqnarray}
&&\Prob\Big(\Big|\frac1N\sum_i\psi_i-\E\big[\tfrac1N\sum_i \psi_i\big]\Big|>t\Big)\le\cr
&&2\exp\left(-\frac{2t^2N^2}{N\cdot O(1)}\right)=2\exp(-c\,Nt^2),
\end{eqnarray}
which gives \eqref{eq:L6}. 
\end{proof}

One has that $(m_N(z))_N$ is bounded (by $1/\eta$), for each fixed $z\in\C^+$. Hence it is tight. One can similarly show by McDiarmid's inequality that $m_n(z)-\mathbb{E}m_N(z)\to 0 $ in probability. Thus every limit in law of $m_N$ is deterministic.  Fix now a subsequence along which $m_N^{(1)}(z)\to\tilde m(z)$ in probability, for some limit $\tilde m(z)$ with $|\tilde m(z)|\le1/\eta$. 

Our next aim is to compute $\lim_N\E[ \psi_1]$.
\begin{lemma}\label{lem:8}
Along such a subsequence,
\begin{equation}\label{eq:dagger}
\E[ \psi_1]\ \longrightarrow\ \tilde m(z)\int\frac{t\,dH(t)}{1+\gamma t\tilde m(z)}.
\end{equation}
\end{lemma}

\begin{proof}
Since $m_N^{(1)}(z)$ is $\mathcal F^{(1)}$-measurable and $w_1$ is $\sigma(Y_1)$-measurable, $w_1\perp\!\!\!\perp m_N^{(1)}(z)$ exactly. This implies that 
$(w_1, m_N^{(1)}(z))\Longrightarrow (w, \tilde m(z))$ where $w\sim H$ is independent of $\tilde m(z)$. 
Note that $h(w, m)$ is a bounded continuous function if $m$ is the Stieltjes transform of a probability distribution $\nu$ and $w$ is bounded.  
From the above, we deduce that 
$$\E (\psi_1)\to \int \frac{t \tilde m(z)}{1+\gamma t \tilde m(z)}dH(t).$$ 
%The fact that $\tilde m(z)$
 %is deterministic then allows to conclude. 
\end{proof}
\subsection{Fixed-point argument and conclusion}

\begin{proof}We now show Theorem~\ref{thm: uni}.
 Along a subsequence, by \eqref{eq:L3}, one  has that $m_N^{(1)}(z)\to\tilde m(z)$ in probability. 
By exchangeability of $Y_1,\dots,Y_N$, one has that $\E[\frac1N\sum_i\psi_i]=\E[\psi_1]$. Combining Lemma~\ref{lem:6} and Lemma~\ref{lem:8}, one gets that
\[
\frac1N\sum_i\psi_i\ \xrightarrow{P}\ \tilde m(z)\int\frac{t\,dH(t)}{1+\gamma t\tilde m(z)}.
\]
Substituting into \eqref{eq:starstar}, one obtains that
\[
z\,\tilde m(z)=\tilde m(z)\int\frac{t\,dH(t)}{1+\gamma t\tilde m(z)}-1,
\]
that is,
\[
-\frac1{\tilde m(z)}=z-\int\frac{t\,dH(t)}{1+\gamma t\tilde m(z)}\quad\Longleftrightarrow\quad \tilde m(z)=-\Big(z-\int\frac{t\,dH(t)}{1+ \gamma t\tilde m(z)}\Big)^{-1}.
\]
Moreover $\tilde m(z)$ arises as a subsequential limit (in probability, hence along a further a.s. convergent subsequence) of $m_N(z)$, itself the Stieltjes transform of the probability measure $\mu_{M_N}$. One has that $|M_N|_{\op}\leq C_0$ with probability tending to $1$ by Theorem 4.4.3 in \cite{Vershynin}. By tightness, the class of Stieltjes transforms of probability measures on $\R$ is closed under such limits: $\tilde m(z)$ is itself the Stieltjes transform of a probability measure. 

One now uses the classical uniqueness of the solution in $\C^+$ of the fixed-point equation (\ref{eq:main}) (see \cite{SilversteinChoi1995}, \cite{BaiSilverstein}). Since every subsequence of $(m_N(z))_N$ has a further subsequence converging in probability to the same limit $m(z)$, the whole sequence converges: $m_N(z)\xrightarrow{P}m(z)$, for every fixed $z\in\C^+$. 

One can then use standard arguments to conclude that $\mu_{M_N}\to\mu$ weakly in probability.
\end{proof}

\end{document}